\documentclass{amsart}
\usepackage{comment}
\usepackage{amsmath, amssymb, amsgen, amsthm, amscd, xspace}

\newcommand{\R}{{\mathbb{R}}}

\newcommand{\g}{{\mathfrak g}}

\newcommand{\fa}{{\mathfrak a}}
\newcommand{\fb}{{\mathfrak b}}

\newcommand{\ff}{{\mathfrak f}}

\newcommand{\fh}{{\mathfrak h}}
\newcommand{\fk}{{\mathfrak k}}

\newcommand{\ft}{{\mathfrak t}}

\newcommand{\fsl}{{\mathfrak s}{\mathfrak l}}

\newcommand{\ad}{{\mathrm a}{\mathrm d\,}}

\newcommand{\Id}{{\mathrm I}{\mathrm d\,}}

\newcommand{\Int}{{\mathrm I}{\mathrm n}{\mathrm t\,}}
\newcommand{\Is}{{\mathit I}}
\newcommand{\Ker}{{\mathrm K}{\mathrm e}{\mathrm r\,}}

 \makeatother

\newtheorem{thm}{Theorem}[section]

\newtheorem{prop}[thm]{Proposition}

\newtheorem{lem}[thm]{Lemma}

\newtheorem{ex}[thm]{Example}

\begin{document}

\title[On isometry groups of pseudo-Riemannian compact Lie groups]{On isometry groups of pseudo-Riemannian compact Lie groups}
\author{Fuhai Zhu}
\address{Department of Mathematics, Nanjing University, Nanjing 210023, P.R. China}
\email{zhufuhai@nju.edu.cn}

\author{Zhiqi Chen}
\address{Corresponding author. School of Mathematical Sciences and LPMC, Nankai University, Tianjin 300071, China}
\email{chenzhiqi@nankai.edu.cn}

\author{Ke Liang}
\address{School of Mathematical Sciences and LPMC, Nankai University, Tianjin 300071, China}
\email{liangke@nankai.edu.cn}

%

\begin{abstract}
Let $G$ be a connected, simply-connected, compact simple Lie group. In this paper, we show that the isometry group of $G$ with a left-invariant pseudo-Riemannan metric is compact. Furthermore, the identity component of the isometry group is compact if $G$ is not simply-connected.
\end{abstract}


\maketitle

\section{Introduction}
The study on isometry groups plays a fundamental role in the study of geometry, and there are beautiful results on isometry groups of Riemannian manifolds and Lorentzian manifolds, see \cite{DAm1988,Gor1980,GW1985,OT1976} and the references therein.

But for a general pseudo-Riemannian manifold $M$, we know much less than that for Riemannian and Lorentzian cases. Let $\Is(M)$ be the isometry group of $M$ and $\Is_o(M)$ the identity component of $\Is(M)$. An important result is given by Gromov in \cite{Gro1987} that $\Is(M)$ contains a closed connected normal abelian subgroup $A$ such that $\Is(M)/A$ is compact if $M$ is compact, simply-connected and real analytic. Moreover, D'Ambra (\cite{DAm1988}) proves that the isometry group $\Is(M)$ of a compact, real analytic, simply connected Lorentzian manifold is compact, and gives an example of a simply connected compact pseudo-Riemannian manifold of type $(7,2)$ with a noncompact isometry group.

In this paper, we focus on group manifolds, and prove the compactness of the isometry group of a connected, compact simple Lie group with a left-invariant pseudo-Riemannian metric.

The paper is organized as follows. In section 2, we will give some notations and list some important results on isometry groups which are useful tools to prove the following theorems in section 3.
\begin{thm}\label{Compact}
If $G$ be a connected, simply-connected, compact simple Lie group with a left-invariant pseudo-Riemannian metric, then the isometry group $\Is(G)$ of $G$ is compact.
\end{thm}

Furthermore, assume that $G$ is a connected compact simple Lie group which is not necessarily simply-connected. Let $\widetilde{G}$ be the universal cover of $G$ with the covering map $\pi:\widetilde{G}\rightarrow G$. Then we have
\begin{thm}\label{nsc}
For any left-invariant pseudo-Riemannian metric on $G$, $\Is_o(G)$ is compact and $\Is_o(G)=\Is_o(\widetilde{G})/\Gamma$, where $\Gamma=\{L_a|a\in\Ker\pi\}$.
\end{thm}

\section{Results on isometry groups}

Let $G$ be a connected Lie group. For any $y\in G$, the left (resp. right) translation $x\mapsto yx$ (resp. $x\mapsto xy^{-1}$) of $G$ onto itself is denoted by  $L_y$ (resp. $R_y$). The set $L(G)$ (resp. $R(G)$) of all left (resp. right) translations of $G$ is a  Lie transformation group of $G$ and the  map: $y\mapsto L_y$ (resp. $y\mapsto R_y$) gives an isomorphism of $G$ onto $L(G)$ (resp. $R(G)$).

If $G$ is endowed with a left-invariant Riemannian metric, then the group $\Is(G)$ of all isometries contains $L(G)$. If $\Is(G)$ contains  $R(G)$, the metric is bi-invariant. It is well-known that if $G$ is connected, compact and semisimple with a bi-invariant metric, then the identity component $\Is_o(G)$ of $\Is(G)$ is  $L(G)R(G)$. In general, if the Riemannian metric is left-invariant but not bi-invariant, Ochiai and Takahashi proved the following theorem for simple Lie groups.

\begin{thm}[{\cite[Theorem~1]{OT1976}}] Let $G$ be a connected, compact simple Lie group with a left-invariant Riemannian metric. Then $\Is_o(G)$ is contained in $L(G)R(G)$, that is, for each isometry $f$ in $\Is_o(G)$ there exist $x, y$ in $G$ such that $f=L_x\circ R_y$.
\end{thm}

The key ingredient of the proof is the following theorem, which we quoted here for the readers' convenience since it will be used frequently in this paper.
\begin{thm}[{\cite[Theorem~4]{OT1976}}]\label{Normal}
Let $K$ be a connected, compact Lie group. Let $G$ and $H$ be closed, connected subgroups of $K$ such that
\begin{enumerate}
\item $G$ is simple.
\item $K=GH$ and $G\cap H = \{e\}$.
\item $H$ contains no normal subgroups of $K$ except $\{e\}$.
\end{enumerate}
Then $G$ is a normal subgroup of $K$.
\end{thm}

If the manifold is not a group manifold but simply-connected and analytic, then Gromov gave the following result.
\begin{prop}[{\cite[0.6.B]{Gro1987}}]\label{compact-quotient}
Let $M$ be a simply-connected, compact, real analytic manifold with a pseudo-Riemannian metric. Then
\begin{enumerate}
\item $\Is(M)$ has finitely many connected components.
\item $\Is(M)$ contains a closed connected normal abelian subgroup $A$ such that $\Is(M)/A$ is compact.
\item If $x\in \Is(M)$ satisfying $xv=v$ and $D_x(T_vM)=id$, then $x=id$.
\end{enumerate}
\end{prop}

Furthermore, for Lorentzian manifolds, D'Ambra proved the following theorem.
\begin{thm}[{\cite[Theorem~1.1]{DAm1988}}] If $M$ is a compact, real analytic, simply connected Lorentzian manifold, then the
isometry group $\Is(M)$ is compact.
\end{thm}

The following Lemma is useful in the paper of D'Ambra.
\begin{lem}[{\cite[Lemma~3.1]{DAm1988}}]\label{torus-orbit}
Let $M$ be a compact, simply connected pseudo-riemannian real analytic
manifold. Then the orbit of any maximal connected abelian subgroup $A\subset \Is(M)$ in $M$ equals to that of the maximal torus $T\subset A$.
\end{lem}



But in general, $\Is(M)$ is not necessarily compact.

\begin{ex}[{\cite[P.~556]{DAm1988}}]
Let $M=SL(2,\R)/\Gamma$ where the group $SL(2,\R)$ is
endowed with the $(1, 2)$ metric corresponding to the Killing form on the Lie
algebra $\fsl(2, \R)$ and $\Gamma$ is a cocompact lattice in $SL(2,\R)$. The isometry group
$\Is(M)$ equals $SL(2,\R)$.
\end{ex}

\section{The proofs of Theorems~\ref{Compact} and~\ref{nsc}}

Let $\Is_o(G)$ be the identity component of the isometry group $\Is(G)$ of $G$. Since the metric is left-invariant, $L(G)$, the group of left-multiplications of $G$, is a subgroup of $\Is_o(G)$. Then $$\Is_o(G)=L(G)H_0,$$ where $H_0$ is the identity component of the isotropic subgroup $H$ at the identity $e\in G$.
Note that $L(G)\cap H=\{\Id_G\}$. 
By Proposition~\ref{compact-quotient}, there exists a closed connected normal abelian subgroup $A$ of $\Is_o(G)$ such that
$\Is_0(G)/A$ is compact. Therefore, by \cite[Theorem~3.7 in Chapter~XV]{Hoc1965},
$$\Is_o(G)=KA,$$ where $K$ is a maximal compact subgroup containing $L(G)$. One may choose $A$ to be simply-connected, since any connected abelian Lie group is the direct product of the maximal torus and a simply-connected abelian subgroup, which can be chosen to be a normal subgroup of $\Is_o(G)$. Our main purpose is to prove that $A$ is trivial. It is necessary to investigate the structure of $K$ and $A$ respectively.

\subsection{The structure of $K$}
The following easy lemma is crucial and useful.
\begin{lem}\label{lem-rm}
Let $\rho\in \Is(G)$. If $\rho L_g=L_g\rho$ for any $g\in G$, then $\rho\in R(G)$, where $R(G)$ is the group of right-translations of $G$.
\end{lem}
\begin{proof}
For any $x\in G$, $\rho L_g(x)=\rho(gx)=g\rho(x)$, for any $g\in G$. Taking $x=e$, we have that $\rho(g)=g\rho(e)=R_{\rho(e)}(g)$.
\end{proof}

\begin{prop}
  $K\subset L(G)R(G)$.
\end{prop}
\begin{proof}
Let $H_c=K\cap H$. Then $K=L(G)H_c$. By Theorem~\ref{Normal}, $L(G)$ is a normal subgroup of $K$. Then there exists a compact normal subgroup $G_1$ such that $K=L(G)G_1$ and the intersection $G_1\cap L(G)$ is finite. Thanks to the above Lemma, we have $G_1\subset R(G)$.
\end{proof}

Since $G_1$ is the product of a semisimple normal subgroup and its center, there exist a semisemiple subgroup $G_2$ and a torus $Z$ of $G$ such that $G_1=R(G_2)R(Z)$. Denote by $\fk,\g_0,\g_1,\fh,\fh_c,\fa$ the Lie algebras of $K,G,G_1,H,H_c,A$, respectively.
We have that $\fk=\g_0\oplus\g_1$ as a direct sum of compact ideals. For any $x\in \fh_c$, there exist unique $x_0\in\g_0$, $x_1\in\g_1$ such that $x=x_0+x_1$. Thus we have two maps $\Delta_i:\fh_c\rightarrow \g_i$, $i=0,1$, defined by $$\Delta_i(x)=x_i.$$
It is easy to see that $\Delta_i$ are injective Lie algebra homomorphisms and $\Delta_1$ is an isomorphism. Set $\Delta=\Delta_0\circ\Delta_1^{-1}$.

\subsection{The bilinear form on Lie algebra of $\Is_o(G)$} Since $\Is_o(G)=KA=L(G)H_0$, identifying the Lie algebra of $L(G)$ with the Lie algebra $\g_0$ of $G$, we have that the Lie algebra of $\Is_o(G)$ is
$$\fk\dot+\fa=\g_0\dot+\fh,$$
as a direct sum of vector spaces. The left-invariant metric on $G$ defines on $\g_0$ an $\fh$-invariant bilinear form $\langle\cdot,\cdot \rangle$, which may be extended to $\g_0\dot+\fh$ by defining
$$\langle \g_0\dot+\fh,\fh\rangle=0.$$

\subsection{The structure of $A$}

By Lemma~\ref{lem-rm} again, one can easily see that the only element in $A$ which commutes with $L(G)$ is the identity element. Since $A$ is a normal subgroup of $\Is_o(G)$, so is $C_K(A)$, which is the centralizers of $A$ in $K$.
Therefore, $C_{K}(A)=R(G_3)R(T_1)$, where $G_3$ is a normal subgroup of $G_2$ and $T_1\subset Z$. Let $T_3$ be a maximal torus of $G_3$. Then $R(T_3)R(T_1)A$ is a maximal connected abelian subgroup of $\Is_o(G)$. By Lemma~\ref{torus-orbit}, the orbit of this group through the identity element is the same as that of $R(T_3)R(T_1)$. Thus $A\cdot e\subset T_3T_1$, which is independent of the choice of $T_3$. Therefore, one can easily see that $A\cdot e\subset T_1$, which defines a map $\varphi:A\rightarrow T_1$ by $$\varphi(a)=a'=a\cdot e.$$

\begin{lem}
The map $\varphi$ is a group homomorphism.
\end{lem}
\begin{proof}
For any $a, b\in A$, $$\varphi(ab)=(ab)\cdot e=a(\varphi(b))=a(R_{\varphi(b)^{-1}}e)=R_{\varphi(b)^{-1}}a(e)=\varphi(a)\varphi(b),$$ since $A$ commutes with $R_{T_1}$.
\end{proof}

Identify the Lie algebra of $R_{T_1}$ with $\ft_1$, the Lie algebra of $T_1$. For any $b\in \mathfrak a$, there exists $b_1\in\mathfrak t_1$ such that $b+b_1\in \mathfrak h$. Note that the elements in $L(G)$ and $R(G)$ commute, we have $[x,b_1]=0$, for any $x$ in $\g_0$, the Lie algebra of $L(G)$. Therefore, for any $u,v\in \g_0$, $$\langle [b+b_1,u],v\rangle+\langle u,[b+b_1,v]\rangle=0,$$
which implies that $$\langle [b,u],v\rangle+\langle u,[b,v]\rangle=0.$$

Let $\ft_0=\Delta(\ft_1)\subset\g_0$. Since $\ft_0$ is a compactly embedded subalgebra of $\g_0\dot+\fh$, for any $t_0\in\ft_0$, the eigenvalues of $\ad t_0$ are $0$ or purely imaginary. As a $\ft_0$-module, $$\mathfrak a=\mathfrak a_0+\mathfrak a_1,$$ where $\fa_0$ is the trivial submodule and $\fa_1$ is the direct sum of irreducible 2-dimensional submodules.

\begin{lem}
$\fa_1\subseteq\fh$.
\end{lem}
\begin{proof}
For any $a\in \mathfrak a_1$, there exist $a_0,t\in \ft_0$ such that $a+a_0\in \mathfrak h$ and $[t,a]\neq 0$. Since $[t,a_0]=0$, we have $[t,a]\in\mathfrak h$. It implies that $a\in \mathfrak h$, i.e. $\mathfrak a_1\subseteq \mathfrak h$.
\end{proof}

\begin{lem}
$\fa_0=\{0\}$.
\end{lem}
\begin{proof}
Let $\fb_1=d\varphi(\fa_0)\subseteq\g_1$ and $\fb_0=\Delta(\fb_1)\subseteq \g_0$. Set $\ff=[\fb_0,\g_0]$.
Then we have
$$\g_0=C_{\g_0}(\fb_0)\dot+\ff.$$
It is easy to see that $[\fa_0,\ff]\subseteq \fa_1$, $[\fa_0,C_{\g_0}(\fb_0)]\subseteq \fa_0$. Since $\langle\cdot ,\cdot \rangle$ is $\fh$-invariant, it is not hard to see that it is $\fb_0$-invariant. Therefore, $\langle\cdot,\cdot \rangle$ is nondegenerate when restricted to $C_{\g_0}(\fb_0)$ or $\ff$.

For any $a\in \mathfrak a_0$, if there exists $x\in C_{\g_0}(\fb_0)$ such that $[a,x]\notin\fh$, then there exists $y\in C_{\g_0}(\fb_0)$ such that $$\langle [a,x],y\rangle\neq 0.$$
There exist $x_i,y_i\in\ff$, $i=1,2,\cdots,m$, such that $x=\sum_{i=1}^m[x_i,y_i].$ It follows that
\begin{eqnarray*}
\langle [a,x],y\rangle &=& \sum \langle [a,[x_i,y_i]],y\rangle \\
&=& \sum(\langle[[a,x_i],y_i],y\rangle+\langle[x_i,[a,y_i]],y\rangle) \\
&=& \sum(-\langle[[a,x_i],y],y_i\rangle+\langle x_i,[[a,y_i],y]\rangle).
\end{eqnarray*}
Since $a\in\fa_0$, $x_i\in \ff$, $y\in C_{\g_0}(\fb_0)$, we have $[a,x_i], [[a,x_i],y]\in\fa_1\subseteq\fh$. Therefore, $\langle[[a,x_i],y],y_i\rangle=0$. Similarly, $\langle x_i,[[a,y_i],y]\rangle=0$.
It follows that $$\langle [a,x],y\rangle=0,$$ which is a contradiction. Hence
$$[a,C_{\g_0}(\fb_0)]\subseteq\fh,\ \ \ \forall a\in \mathfrak a_0.$$
Therefore,
$$[\fa_0,\g_0]\subseteq\fh.$$
By the facts that for any $a\in \mathfrak a_0$ there exists $a_1\in\ft_1$ such that $a+a_1\in\mathfrak h$ and $[a_1,\g_0]=0$, we have
$$[a+a_1,\g_0]\subseteq\fh.$$
Together with $\exp s(t_a+a)e=e$, we have
$$\exp s(t_a+a)=id.$$
Thus $\exp st_a=-\exp sa$. Then $\{{\mathbb R} t_a\}=\{{\mathbb R} a\}$. Namely $\mathfrak a_0\subset \fb_0$, which is impossible. Thus $\mathfrak a=\mathfrak a_1\subset \mathfrak h$.
\end{proof}

\subsection{The proofs of the main results}

Now we are in a position to prove Theorems~\ref{Compact} and ~\ref{nsc}.

\begin{proof}[\bf The proof of Theorem~\ref{Compact}] Thanks to the above Lemma, $\fa_0=\{0\}$ and $\fa=\fa_1\subseteq\fh$. Since $\mathfrak h$, as the Lie algebra of isotropy subgroup, has no nontrivial ideal of $\g_0\dot+\fh$, we get that $\fa=\{0\}$. Hence $A$ is trivial and $\Is(G)$ is compact. \end{proof}

Now let $G$ be a connected simple compact Lie group, which is not necessarily simply-connected. Let $\widetilde{G}$ be the universal cover of $G$ with the covering map $\pi:\widetilde{G}\rightarrow G$.

\begin{proof}[\bf The proof of Theorem~\ref{nsc}] Set $n=|\Ker\pi|$. Each $\varphi\in\Is_o(G)$ can be lifted to $n$ isometries $\varphi_1,\ldots,\varphi_n$ in $\Is_o(\widetilde{G})$ and $L_G$ is lifted to $L_{\widetilde{G}}$. Let $\pi:\widetilde{G}\rightarrow G$ be the natural projection with $Z=\Ker \pi$. Let $G_1$ be the closed subgroup of $\Is_o(\widetilde{G})$ which preserves the coset of $Z$. Then $G_1$ consists of all lifts of the elements of $\Is_o(G)$. Therefore,
$\Is_o(G)$ is compact since $\Is_o(\widetilde{G})$, hence $G_1$, is compact. Since $L_{\widetilde{G}}\lhd G_1$, we have
$L(G)\lhd \Is_o(G)$. Hence $\Is_o(G)\subset L(G)R(G)$. Write $\Is_o(G)=L(G)H$ and $\Is_o(\widetilde{G})=L({\widetilde{G}})\widetilde{H}$.
It is easy to see that $$H=\{\varphi\in\Int(\g)|\langle\varphi(X),\varphi(Y)\rangle=\langle X,Y\rangle,\forall X,Y\in\g\}=\widetilde{H}.$$
Then Theorem~\ref{nsc} holds. \end{proof}

\section{Acknowledgements} This work was partially supported by National Natural Science Foundation of China (11571182 and 11931009) and Natural Science Foundation of Tianjin (19JCYBJC30600).

\end{document}